\documentclass[11pt]{article}

	\usepackage{makecell}
	\usepackage[margin=1in]{geometry}  
	\usepackage{graphicx}              
	\usepackage{amsmath}               
	\usepackage{amsfonts}              
	\usepackage{amsthm}                
	\usepackage{mathtools}
	\usepackage{xcolor}
	\usepackage{hyperref}
	\usepackage{MnSymbol}
	\newtheorem{thm}{Theorem}[section]
	\newtheorem{lem}[thm]{Lemma}

	\newtheorem{defn}{Definition}
	
	\newtheorem{remark}{Remark}

\begin{document}

		\nocite{*}
		
		\title{ On The Ellipticity of Generalised Monge-Amp\`ere Equations on Vector Bundles}
		
		\author{Gao Chen and Kartick Ghosh} 
	
	\maketitle

	\begin{abstract}
		In this paper, we study the ellipticity of the vector bundle versions of the Monge-Amp\`ere, $J$, dHYM and $\sigma_{k}$-equations at a point. These are nonlinear geometric partial differential equations defined on a holomorphic vector bundle over a compact K\"ahler manifold. We show that when both the dimension of the manifold and the rank of the bundle are greater than or equal to three, these equations do not preserve ellipticity along continuity paths in the connected component of the trivial solution. However, the $\sigma_{2}$-equation does preserve ellipticity along continuity paths.
	\end{abstract}
	
	\section{Introduction}
    One of the fundamental theorems in geometric analysis is the Calabi-Yau theorem \cite{YauCalabiconjecture}. It states that on any compact K\"ahler manifold $X$ with a K\"ahler form $\omega_0$, for any smooth function $f$, there exists a function $\varphi$ such that $\omega_\varphi=\omega_0+\sqrt{-1}\partial\bar\partial\varphi$ satisfies $\omega_\varphi^n=c e^f\omega_0^n$ where $c=\frac{\int_X \omega_0^n}{\int_X e^f\omega_0^n}$. Yau's proof is based on the continuity method. He considered $\omega_t^n=c_t e^{tf}\omega_0^n$ where $\omega_t=\omega_0+\sqrt{-1}\partial\bar\partial\varphi_t$, and $c_t=\frac{\int_X \omega_0^n}{\int_X e^{tf}\omega_0^n}$. The key steps are as follows. Let $I$ be the set of $t$ such that $\varphi_t$ solves the equations. Suppose that $t_i\in I$ converges to $t_0$. Yau proved an a priori estimate  $||\varphi_{t_i}||_{C^{k,\alpha}}\le C(k,\alpha)<\infty$. This implies that $\omega_{t_i}$ has a subsequence converging to $\omega_{t_0}$. A priori, $\omega_{t_0}$ is only the limit of K\"ahler metrics and therefore may fail to be K\"ahler. However, the equation forces $\omega_{t_0}$ to remain K\"ahler. We say that the K\"ahler condition is preserved along the continuity path. This is essential for proving the closedness of $I$. Moreover, the K\"ahler condition implies that the equation is elliptic which is essential for proving the openness of $I$. 
\par 
	Another fundamental theorem in geometric analysis is the Donaldson-Uhlenbeck-Yau theorem (\cite{Donaldson1},\cite{Donaldson2},\cite{Uhlenbeck1986}). It states that a holomorphic vector bundle over a compact K\"ahler manifold admits a Hermitian-Einstein metric if and only if the bundle is slope polystable. This theorem can be seen as a bridge between differential geometry and algebraic geometry, with the equation originating in mathematical physics. Consequently, it has inspired numerous research directions in differential geometry, algebraic geometry, and mathematical physics. \par
    Uhlenbeck and Yau's proof was based on the continuity method while Donaldson's proof was based on geometric flow. There is also a geometric flow approach to Calabi's conjecture due to Cao \cite{Caodeformation}. He proved that when $c_1(X)=0$, the Ricci flow starting from $\omega_0$ converges to the Calabi-Yau metric. A key part of the proof is that the K\"ahler condition is preserved along the flow. Moreover, the K\"ahler condition implies the parabolicity of the flow. \par
	Due to the great success of the Calabi-Yau theorem and the Donaldson-Uhlenbeck-Yau theorem, many new equations have recently been proposed  (\cite{Pingaliadv},\cite{collins2018},\cite{Takahashi2024},\cite{Dervan2024},\cite{Zhang2022},\cite{Tipleretal}), some of which may be related to Bridgeland's categorical stability conditions (\cite{Bridgeland}). In fact, the stability conditions for the $Z$-critical equations introduced by Dervan-McCarthy-Sektnan (\cite{Dervan2024}) are expected to be Bridgeland stability conditions. Near the large-volume limit for the $Z$-critical equations (see theorem $1.1$ in \cite{Dervan2024}), a Donaldson-Uhlenbeck-Yau theorem has been proposed and proved. More generally, suppose that $(h_0,\bar{\partial}_{0})$ solves the equation and satifies some positivity condition, then the Donaldson-Uhlenbeck-Yau theorem also holds near $(h_0,\bar{\partial}_{0})$  (\cite{delloque2025}, \cite{Scarpa}). One reason is that the ellipticity condition is an open condition and the posititiy of $h_0$ implies the ellipticity of $h_0$. \par 
     Due to the success of these results to solve for $h$ near $h_0$, it is natural to ask what happens if $h$ is away from $h_0$. In this case, we can no longer use the ellipiticity of $h_0$, and the openness of the ellipticity condition. A natural idea is to use either the continuity method or the flow method. To use the continuity method, we should expect a positivity condition $K$, which is analogous to the K\"ahler condition in Yau's solution of the Calabi conjecture. The desired property is that if $h_{t_i}$ solves the equation and satisfies condition $K$, then its limit also satisfies the condition $K$. Moreover, the condition $K$ should imply ellipticity. To use the flow method, one would similarly expect a positivity condition $K$  that is preserved along the flow and implies parabolicity. \par
  
     Note that the condition $K$ is expected to be a pointwise condition, so we restrict ourselves to a single point. The theorems in this paper show that when the dimension and the rank are at least three, both the continuity method and the flow method fail for the vector bundle Monge-Amp\`ere equation, the $\sigma_{k}$ ($k\ge 3$) equations, the vector bundle $J$-equation and the deformed Hermitian-Yang-Mills equation near the small-volume limit. More precisely, we prove that there exists a continuous path of solutions $F_{H_t}, t\in[0,1]$ such that $F_{H_0}$ is the trivial solution $\rho\sum_{i=1}^{n}\mathrm{Id}\sqrt{-1}dz^{i}\wedge d\bar{z}^{i}$ (for some positive real number $\rho$) while $F_{H_1}$ does not satisfy the ellipticity condition. Indeed, if such a condition $K$ existed, then $F_{H_{t}}$ would satisfy condition $K$ for all $t$, which would imply the ellipticity of $F_{H_{1}}$. The significance of the  connected component of trivial solution is also emphasized in \cite{Lin2023} for the line bundle case. In a helpful discussion with Yau, he suggested that we should make an additional assumption that the positivity condition for the deformed Hermitian-Yang-Mills equation holds along a continuity path instead of getting it for free from the equation.\par

	Before stating our main results in details, we introduce the relevant equations. Suppose $(E,H)$ is a holomorphic vector bundle over an $n$-dimensional compact K\"ahler manifold $(X,\omega)$. We denote the curvature of the Chern connection of $H$ on $E$ as $F_{H}$. The vector bundle Monge-Amp\`ere(vbMA) equation (\cite{Pingaliadv}) asks for a Hermitian $H$ such that 
	\[(\sqrt{-1}F_{H})^{n}=\eta\otimes\mathrm{Id},\]
	where $\eta$ is a given volume form. The $\sigma_{k}$ equations (\cite{Zhang2022}) asks for a Hermitian metric $H$ which satisfies 
	\[(\sqrt{-1}F_{H})^{k}\wedge\omega^{n-k}=\eta\otimes\mathrm{Id},\]
	where $\eta$ is a given volume form. The vector bundle $J$-equation (\cite{Takahashi2024}) is defined for bundles with the properties $ch_{n}(E)>0,[\omega]\cdot ch_{n-1}(E)>0$. The vector bundle $J$-equation asks for a metric $H$ such that 
	\[c(\sqrt{-1}F_{H})^{n}-\omega\mathrm{Id}\wedge(\sqrt{-1}F_{H})^{n-1}=0,\]
	for some constant $c=\frac{[\omega]\cdot ch_{n-1}(E)}{nch_{n}(E)}>0$. This equation can be seen as inverse $\sigma_{k}$ equation. The deformed Hermitian-Yang-Mills equation(dHYM) for vector bundles proposed by Collins-Yau (\cite{collins2018}) asks for a metric $H$ such that 
	\[\mathrm{Im}\left(e^{-\sqrt{-1}\hat{\theta}}\big(\omega\otimes\mathrm{Id}_{E}-F_{H}\big)^{n}\right)=0\]
	with 
	\[\int_{X}\mathrm{Tr}\left(\omega\otimes\mathrm{Id}_{E}-F_{H}\right)^{n}\in \mathbb{R}_{>0}e^{\sqrt{-1}\hat{\theta}}\]
	and the imaginary part is defined using the metric $H$. The dHYM equation is an $Z$-critical equation. Hence we study the ellipticity of this equation near the small-volume limit.\par
 Our main results are the following.
	\begin{thm}
		\label{thmvbma}
		The ellipticity is not preserved for the vector bundle Monge-Amp\`ere equation at a point along continuity paths in the connected component of trivial solution when $rank(E)\ge 3$ and $\dim(X)\ge 3$.
	\end{thm}
	\begin{remark}
		Pingali showed in \cite{Pingaliadv} that the ellipticity is preseved for the vector bundle Monge-Amp\`ere equation along continuity paths when $dim(X)=2$ (see lemma $2.5$ in \cite{Pingaliadv} and the discussion thereafter). As all the other equations are equivalent to the vector bundle Monge-Amp\`ere equation in dimension two by completing squares, they also preserve ellipticity in dimension two along continuity paths.
	\end{remark}
		\begin{thm}
			\label{thmsigma}
		The $\sigma_{k}$ equations do not preserve ellipticity at a point along continuity paths in the connected component of trivial solution when $k\ge 3$ and $rank(E)\ge 3$. In contrast, the $\sigma_{2}$ equation preserves ellipticity along continuity paths.  
	\end{thm}
		\begin{thm}
			\label{thmj}
		The ellipticity is not preserved for the vector bundle $J$-equation at a point along continuity paths in the connected component of trivial solution when $rank(E)\ge 3$ and $\dim(X)\ge 3$.
	\end{thm}
		\begin{thm}
			\label{thmdhym}
		The ellipticity is not preserved for the dHYM equation at a point along continuity paths in the connected component of trivial solution when $rank(E)\ge 3$ and $\dim(X)\ge 3$ near the small-volume limit (see section \ref{section dhym} for details).
	\end{thm}
	In section \ref{section 2} we give some preliminaries. In section \ref{section vbma} we prove theorem  \ref{thmvbma}. In section \ref{sectionsigmak} and \ref{sectionsigma2}, we prove theorem \ref{thmsigma}. In section \ref{section J-equation} we prove theorem \ref{thmj}. In section \ref{section dhym} we prove theorem \ref{thmdhym}.\\ \\
	\textbf{Acknowledgements:} The authors would like to thank Ruadha\'i Dervan, Vamsi Pingali, Lars Martin Sektnan, Carlo Scarpa, and Shing-Tung Yau for making helpful comments on the first draft. The second author was supported by a postdoctoral fellowship from the Institute of Geometry and Physics at the University of Science and Technology of China.
	\section{Preliminaries}
	\label{section 2}
	Suppose $E$ is a holomorphic vector bundle over a compact K\"ahler manifold $X$ and $\eta$ is a given volume form. Then the vector bundle Monge-Amp\`ere equation (\cite{Pingaliadv}) asks for a metric $H$ on $E$ such that 
	\[(\sqrt{-1}F_{H})^{n}=\eta\otimes \mathrm{Id}.\]
First let us recall the definition of ellipticity.
\begin{defn}[ See Appendix $III$ in \cite{Donaldson-Kronheimer}]
	 Let $D:\Gamma(V_{1})\rightarrow \Gamma(V_{2})$ be a linear differenial operator of order $l$ between sections of bundles over a compact base manifold $X$. The highest order term in $D$ defines the `Symbol' $\sigma_{D}$; for each cotangent vector $\chi$ at a point $x\in X$ the symbol gives a linear map $\sigma_{D}(\chi)$ from the fibre of $V_{1}$ at $x$ to that of $V_{2}$. The operator is elliptic if $\sigma_{D}(\chi)$ is invertible for all non-zero $\chi$.	 
	\end{defn}
	If $H_{0}$ is some Hermitian metric on $E$, then any other Hermitian
	metric is of the form $H =H_{0}e^{G}$, where $G\in End(E,H_{0})$, i.e. $G$ is a $H_{0}$-Hermitian
	section of $End(E)$. The linearisation of the vbMA equation at $H_{0}$ i.e., $L_{H_{0}}: End(E,H_{0})\rightarrow \Omega^{(n,n)}(End(E))$ is given by 
	\[\Psi\mapsto \sum_{j=0}^{n-1}(\sqrt{-1}F_{H_{0}})^{j}\wedge\sqrt{-1}\bar{\partial}\partial_{H_{0}}\Psi\wedge(\sqrt{-1}F_{H_{0}})^{n-1-j}.\]
	\begin{lem}[Proposition $1$ in \cite{ballal-pingali}]
		Using a pairing, one can identify $\Omega^{(n,n)}(End(E))$  with $End(E)$. Furthermore, if $H_{0}$ solves the vbMA equation then $L_{H_{0}}\Psi\in End(E,H_{0}) $. Hence we can think that the linearisation is an operator on $End(E,H_{0})$.
	\end{lem}
	The principal symbol can be seen as 
	\[\sigma_{L_{H_{0}}}(\chi)(\Psi)=\sum_{j=0}^{n-1}(\sqrt{-1}F_{H_{0}}(p))^{j}\wedge\sqrt{-1}\bar{\chi}\chi\Psi\wedge(\sqrt{-1}F_{H_{0}}(p))^{n-1-j}.\]
	To define ellipticity we need that $\sigma_{L_{H_{0}}}(\chi)$ is invertible for all non-zero $\chi$. For our computational purposes we introduce the following positivtiy condition $
E$ which should be implied by any reasonable positivity condition $K$. 
 \\
	\textbf{Condition $E$ for the vbMA equation:}
		Suppose $p\in X$ is a point and $H_{0}$ is a metric on the bundle $E$. The condition is that for all endomorphism valued $(1,0)$ form of the type $\xi=a\otimes g$ where $0\ne a\in\Omega^{(1,0)}(X_{p})$ and $0\ne g\in End(E_{p})$, we have 
		\[\mathrm{Tr}\left(\sum_{j=0}^{n-1}\sqrt{-1}\xi\wedge(\sqrt{-1}F_{H_{0}}(p))^{j}\wedge \xi^{\dagger}\wedge(\sqrt{-1}F_{H_{0}}(p))^{n-1-j}\right)>0.\]
Similarly as above, we get the analogous condition $E$ for the vector bundle $J$-equation.\\ \\
\textbf{Condition $E$ for the vector bundle $J$-equation:}
	Suppose $p\in X$ is a point and $H_{0}$ is a metric on $E$. The condition $E$ for the vector bundle $J$-equation is that for all endomorphism valued $(1,0)$ form of the type $\xi=a\otimes g$ where $0\ne a\in\Omega^{(1,0)}(X_{p})$ and $0\ne g\in End(E_{p})$, we have 
	
	\begin{align*}
	&c\mathrm{Tr}\left(\sum_{j=0}^{n-1}\sqrt{-1}\xi\wedge(\sqrt{-1}F_{H_{0}}(p))^{j}\wedge \xi^{\dagger}\wedge(\sqrt{-1}F_{H_{0}}(p))^{n-1-j}\right)\\
	&-\mathrm{Tr}\left(\sum_{j=0}^{n-2}\sqrt{-1}\xi\wedge(\sqrt{-1}F_{H_{0}}(p))^{j}\wedge \xi^{\dagger}\wedge(\sqrt{-1}F_{H_{0}}(p))^{n-2-j}\right)\wedge \omega>0.
	\end{align*}

Similarly, the Condition $E$ for other equations (the $\sigma_{k}$, the dHYM) can be defined but we will not do it here. Now we will prove some lemmas showing the relationship between ellipticity and condition $E$. We will only consider the vbMA case, the other cases ($J$-equation, $\sigma_{k}$-equation, dHYM equation) are similar. 
	\begin{lem}
		\label{Eimpliesellipticity}
		Condition $E$ implies ellipticity for the vbMA equation.
	\end{lem}
	\begin{proof}
		Suppose $p\in X$ is a point, $H_{0}$ is a metric on the bundle $E$ and they satisfy condition $E$ for the vbMA equation. Now on the contrary, let us assume it is not elliptic at $p$. So there exists a non-zero co-tangent vector $\chi\in\Omega^{1,0}(X_{p})$ and non-zero endomorphism 
		$\Psi\in End(E_{p})$ such that
		\begin{align*}
			&\sigma_{L_{H_{0}}}(\chi)(\Psi)\\
			&=\sum_{j=0}^{n-1}(\sqrt{-1}F_{H_{0}}(p))^{j}\wedge\sqrt{-1}\bar{\chi}\chi\Psi\wedge(\sqrt{-1}F_{H_{0}}(p))^{n-1-j}\\	
			&=0.
			\end{align*}
		Multiplying by $\Psi^{*}$ on the right and taking trace we see that 
		\begin{align*}
			&\mathrm{Tr}\left(\sum_{j=0}^{n-1}\sqrt{-1}\xi\wedge(\sqrt{-1}F_{H_{0}}(p))^{j}\wedge \xi^{\dagger}\wedge(\sqrt{-1}F_{H_{0}}(p))^{n-1-j}\right)\\
			&=0,
		\end{align*}
		where $\xi=\chi\Psi^{*}$. This is a contradiction. Hence we are done.
	\end{proof}
	In the reverse direction, we have
	\begin{lem}
		\label{ellipticityimpliesE}
		Suppose $H_{t}$ is a continuous path of metrics for $t\in [0,1]$ along which ellipticity is preserved. If $H_{0}$ satisfies condition $E$, then $H_{t}$ satisfies condition $E$ for all $t\in [0,1]$.
	\end{lem}
	\begin{proof}
		On the contrary, we assume that $t_{0}$ is the first time where the condition $E$ is violated. Because of continuity it is positive semidefinite  at $t_{0}$. Hence there is $\xi=\chi\Psi^{*}$(with $\chi\ne 0$, $\Psi^{*}\ne 0$) such that 
		\begin{align*}
			&\mathrm{Tr}\left(\sum_{j=0}^{n-1}\sqrt{-1}\xi\wedge(\sqrt{-1}F_{H_{t_{0}}}(p))^{j}\wedge \xi^{\dagger}\wedge(\sqrt{-1}F_{H_{t_{0}}}(p))^{n-1-j}\right)=0\\
			&\implies \mathrm{Tr}\left(\sum_{j=0}^{n-1}\sqrt{-1}\chi\Psi^{*}\wedge(\sqrt{-1}F_{H_{t_{0}}}(p))^{j}\wedge \bar{\chi}\Psi\wedge(\sqrt{-1}F_{H_{t_{0}}}(p))^{n-1-j}\right)=0\\
			&\implies\mathrm{Tr}\left(\Psi^{*}\sum_{j=0}^{n-1}(\sqrt{-1}F_{H_{t_{0}}}(p))^{j}\wedge\sqrt{-1} \bar{\chi}\chi\Psi\wedge(\sqrt{-1}F_{H_{t_{0}}}(p))^{n-1-j}\right)=0\\
			&\implies \sum_{j=0}^{n-1}(\sqrt{-1}F_{H_{t_{0}}}(p))^{j}\wedge\sqrt{-1} \bar{\chi}\chi\Psi\wedge(\sqrt{-1}F_{H_{t_{0}}}(p))^{n-1-j}=0.
		\end{align*} 
		This is a contradiction to ellipticity. So we are done.
	\end{proof}
	Our strategy to produce counter-example is the following. We will construct a continuous path of metrics $H_{t}$ at a point solving the equation such that at $t=0$, condition $E$ will be satisfied but at $t=1$, condition $E$ will not be satisfied and it will satisfy the equation for all $t\in[0,1]$. By lemma \ref{Eimpliesellipticity}, ellipticity holds at $t=0$. If ellipticity holds for all $t\in[0,1]$, then condition $E$ holds for all $t\in[0,1]$ by lemma \ref{ellipticityimpliesE}. But it will be false by our construction. Hence ellipticity will not be preserved at a point along continuity paths. \\
	In the following we show that the trivial solution at a point to the vbMA equation satisfies condition $E$ and consequently elliptic.
\begin{lem}
Suppose $p\in X$ is a point and $H_{0}$ is a metric on the bundle $E$ such that at $p$, the curvature is given by
	\[\sqrt{-1}F_{H_{0}}(p)=\sum_{i=1}^{n}\rho\mathrm{Id}\sqrt{-1}dz^{i}\wedge d\bar{z}^{i},\]
	for some positive real number $\rho$. Then this satisfies the vbMA equation and the condition $E$ for the vbMA equation.
	\end{lem}
\begin{proof}
We see that 
\[(\sqrt{-1}F_{H_{0}}(p))^{n}=\rho^{n}\mathrm{Id}\sqrt{-1}dz^{1}\wedge d\bar{z}^{1}\dots \sqrt{-1}dz^{n}\wedge d\bar{z}^{n}.\] Choosing $\rho$ appropriately we can make sure that the equation is satisfied at $p$. We now show that condition $E$ holds. We calculate 
	\begin{align*}
		&\mathrm{Tr}\left(\sum_{j=0}^{n-1}\sqrt{-1}\xi\wedge(\sqrt{-1}F_{H_{0}}(p))^{j}\wedge \xi^{\dagger}\wedge(\sqrt{-1}F_{H_{0}}(p))^{n-1-j}\right)\\
		&=\rho^{n-1}\mathrm{Tr}\left(\sum_{j=0}^{n-1}\sqrt{-1}\xi\wedge(\sum_{i=1}^{n}\sqrt{-1}dz^{i}\wedge d\bar{z}^{i})^{j}\wedge \xi^{\dagger}\wedge(\sum_{i=1}^{n}\sqrt{-1}dz^{i}\wedge d\bar{z}^{i})^{n-1-j}\right)\\
		&>0.
	\end{align*}
	We are done.
\end{proof}
Similarly we can show that the trivial solution satisfies condition $E$ for other cases ( $J$-equation, $\sigma_{k}$-equations, dHYM equation). 

		\section{Counter Example for the vbMA Equation}
		\label{section vbma}
		In this section we prove theorem  \ref{thmvbma}.
		The condition $E$ for the vbMA equation is a pointwise condition and hence we will consider everything pointwise. First let us consider the case $dim(X)=3$ and $rank(E)=3$. We know that we can have trivialisation where the curvature of the Chern connection of a metric at a point takes the form 
		\[\sqrt{-1}F_{H}=\sum_{i=1}^{3}A_{i}\sqrt{-1}dz^{i}\wedge d\bar{z}^{i}+\sum_{i<j}(B_{ij}\sqrt{-1}dz^{i}\wedge d\bar{z}^{j}+B_{ij}^{*}\sqrt{-1}dz^{j}\wedge d\bar{z}^{i}),\]
		where $A_{i}$ are Hermitian and $B_{ij}$ are any matrix. Hereafter we will write $\sqrt{-1}F_{H}$ instead of $\sqrt{-1}F_{H}(p)$ like as above.
		Then we have the following lemma.
		\begin{lem}
			\label{lemma}
		With the notations as above, we have the following formula at a point ( we are igoring the volume form part $\sqrt{-1}dz^{1}\wedge d\bar{z}^{1}\wedge \sqrt{-1}dz^{2}\wedge d\bar{z}^{2}\wedge \sqrt{-1}dz^{3}\wedge d\bar{z}^{3}$)
		\begin{align*}
			&(\sqrt{-1}F_{H})^{3}\\
			&=\{A_{1},A_{2},A_{3}\}-\{A_{1},B_{23},B_{23}^{*}\}-\{A_{2},B_{13},B_{13}^{*}\}-\{A_{3},B_{12},B_{12}^{*}\}+\{B_{13},B_{12}^{*},B_{23}^{*}\}\\
			&+\{B_{13}^{*},B_{12},B_{23}\},
			\end{align*}
		where $\{A,B,C\}=ABC+ACB+BAC+BCA+CAB+CBA$. 
			\end{lem}
			\begin{proof}
				It is a routine calculation. So we omit it.
			\end{proof}
			Since we are interested in continuity paths, we have to consider a path of metrics. So consider a path of metrics with curvature having the following matrices:
		\begin{equation}
			A_{1}=	\begin{bmatrix}
				p(t) & 0 & 0\\
				0 & q(t) & 0\\
				0 & 0 & r(t)
			\end{bmatrix},
			A_{2}=	\begin{bmatrix}
				s(t) & 0 & 0\\
				0 & u(t) & 0\\
				0 & 0 & v(t)
			\end{bmatrix},
			A_{3}=	\begin{bmatrix}
				x(t) & 0 & 0\\
				0 & y(t) & 0\\
				0 & 0 & z(t)
			\end{bmatrix},
		\end{equation}
		\begin{equation}
			B_{12}=	\begin{bmatrix}
				0 & a(t) & 0\\
				0 & 0 & 0\\
				0 & 0 & 0
			\end{bmatrix},
			B_{13}=	\begin{bmatrix}
				0 & 0 & b(t)\\
				0 & 0 & 0\\
				0 & 0 & 0
			\end{bmatrix}, 	B_{23}=	\begin{bmatrix}
				0 & 0 & 0\\
				0 & 0 & c(t)\\
				0 & 0 & 0
			\end{bmatrix}, 
		\end{equation}
		where $p,q,r,s,u,v,x,y,z,a,b,c:[0,1]\rightarrow\mathbb{R}$ are continuous functions.  
		\begin{lem}
			With the matrices defined above, we have
		\begin{align*}
		&(\sqrt{-1}F_{H_{t}})^{3}\\
	&=
		\begin{bmatrix}
				\makecell{6psx-b^{2}(2s+v)\\-a^{2}(2x+y)+2abc} & 0 & 0\\
				0 & \makecell{6quy-c^{2}(2q+r)\\-a^{2}(x+2y)+2abc} & 0\\
				0 & 0 & \makecell{6rvz-c^{2}(q+2r)\\-b^{2}(s+2v)+2abc}
			\end{bmatrix}
			\end{align*} 
				\end{lem}
				\begin{proof}
					This follows easily from the above lemma \ref{lemma} . It is easy to see that $\{A_{1},A_{2},A_{3}\}=\begin{bmatrix}
						6psx & 0 & 0\\
						0  & 6quy & 0\\
						0 & 0 & 6rvz 
					\end{bmatrix}$. Now $B_{23}=cE_{23}$, where $E_{ij}$ is the $3\times 3$ matrix whose $(i,j)$ entry is $1$ and other entries are zero. We calculate
					\begin{align*}
						&\{A_{1},B_{23},B_{23}^{*}\}\\
						&=A_{1}(B_{23}B_{23}^{*}+B_{23}^{*}B_{23})+(B_{23}B_{23}^{*}+B_{23}^{*}B_{23})A_{1}+B_{23}A_{1}B_{23}^{*}+B_{23}^{*}A_{1}B_{23}\\
						&=c^{2}\left(A_{1}(E_{22}+E_{33})+(E_{22}+E_{33})A_{1}+rE_{22}+qE_{33}\right)\\
						&=c^{2}\left((2q+r)E_{22}+(q+2r)E_{33}\right),
					\end{align*}
					similarly we have 
					\[\{A_{2},B_{13},B_{13}^{*}\}=b^{2}\left((2s+v)E_{11}+(s+2v)E_{33}\right)\]
					and 
					\[\{A_{3},B_{12},B_{12}^{*}\}=a^{2}\left((2x+y)E_{11}+(x+2y)E_{22}\right).\]
					We calculate 
					\begin{align*}
						& \{B_{13},B_{12}^{*},B_{23}^{*}\}\\
						&=abc\left(\{E_{13},E_{21},E_{32}\}\right)\\
						&=abc(E_{13}E_{21}E_{32}+E_{13}E_{32}E_{21}+E_{21}E_{13}E_{32}+E_{21}E_{32}E_{13}+E_{32}E_{13}E_{21}+E_{32}E_{21}E_{13})\\
						&=abc(E_{11}+E_{22}+E_{33})
					\end{align*}
					We also have  $\{B_{13}^{*},B_{12},B_{23}\}=(\{B_{13},B_{12}^{*},B_{23}^{*}\})^{*}=(abc(E_{11}+E_{22}+E_{33}))^{*}=abc(E_{11}+E_{22}+E_{33})$. We are done.
				\end{proof}
 To satisfy the the vbMA equation,  we should have  
 \begin{equation}
 	\label{vb123}
 	\begin{cases}
	1)6psx-b^{2}(2s+v)-a^{2}(2x+y)+2abc=K(t)\\
	2)6quy-c^{2}(2q+r)-a^{2}(x+2y)+2abc=K(t)\\
	3)6rvz-c^{2}(q+2r)-b^{2}(s+2v)+2abc=K(t),
	\end{cases}
 \end{equation}
for some positive function $K(t)$ on $[0,1]$.\\
We define
\begin{align*}
	&p(t),s(t),x(t),y(t)=1; \\
	&q(t)=1+t\\
	&r(t)=1+2t\\
	& a(t)=b(t)=c(t)=\sqrt{3t}\\
	&v(t)=1-0.9t
\end{align*} and $u(t),z(t)$ to  be defined. To have $(\sqrt{-1}F_{H_{t}})^{3}=K(t)\mathrm{Id}$, we need to choose $u(t),z(t)$ accordingly.			
From the first equation in (\ref{vb123}),  we get
\begin{align*}
	&K(t)\\
	&=6psx-b^{2}(2s+v)-a^{2}(2x+y)+2abc\\
	&=6-18t+2.7t^{2}+6t\sqrt{3t}
\end{align*}
\begin{lem}
The function $K(t)=6-18t+2.7t^{2}+6t\sqrt{3t}$ is positive on $[0,1]$.
\end{lem}
\begin{proof}
	
	We take the derivative  
	\[K'(t)=-18+5.4t+9\sqrt{3}\sqrt{t}.\]
	Now $K'(0)=-18<0$ and $K'(1)=-12.6+9\sqrt{3}>0$. So $K(t)$ starts decresing and then increases. So there is a minimum. If we put $u=\sqrt{t}$ then the above becomes 
	\[-18+9\sqrt{3}u+5.4u^{2}.\]
	The root in $[0,1]$ is $u_{0}=\frac{-9\sqrt{3}+\sqrt{243+4\times 18\times 5.4}}{10.8}=\frac{-9\sqrt{3}+\sqrt{631.8}}{10.8}$. So we take $t_{0}=u_{0}^{2}$. But at this point $K(t_{0})\approx 0.76>0$. Hence we are done.
\end{proof}
 Now we are ready to define $u(t),z(t)$. From the second equation in (\ref{vb123}), we get  
\begin{align*}
&6quy-c^{2}(2q+r)-a^{2}(x+2y)+2abc=K(t)\\
&\implies 6(1+t)u(t)\times 1-3t(2(1+t)+1+2t)-3t(1+2\times 1)+6t\sqrt{3t}=6-18t+2.7t^{2}+6t\sqrt{3t}\\
&\implies u(t)=\frac{6+14.7t^{2}}{6(1+t)}
\end{align*}		
 This gives us $u(t)$. Now from the third equation in (\ref{vb123}), we get
 \begin{align*}
 	&6rvz-c^{2}(q+2r)-b^{2}(s+2v)+2abc=K(t)\\
 	&\implies 6(1+2t)(1-0.9t)z(t)-3t(1+t+2(1+2t))-3t(1+2(1-0.9t))+6t\sqrt{3t}\\
 	&=6-18t+2.7t^{2}+6t\sqrt{3t}\\
 	&\implies z(t)=\frac{6+12.3t^{2}}{6(1+2t)(1-0.9t)}.
 \end{align*}
 This gives us $z(t)$.\\
	The following are easy to check.\\
$(1)$ At $t=0$, $A_{i}=\mathrm{Id}$ and $B_{ij}=0$.\\
$(2)$ $p,q,r,s,u,v,x,y,z$ stays positive on $[0,1]$ and $p(1)\ne q(1)$.\\
$(3)$ $(\sqrt{-1}F_{H_{t}})^{3}=K(t)\mathrm{Id}$ by construction and $K(t)>0$ on $[0,1]$.\\
$(4)$ $3p(1)x(1)-b(1)^{2}=0$. This will prove that condition $E$ for vbMA equation does not hold at $t=1$.  


Now let us explain how we derived the explicit condition $3p(1)x(1)-b(1)^{2}=0$. First, let us recall the condition $E$ for vbMA equation: at a point $p\in X$ and for a metric $H$, if for all endomorphism valued $(1,0)$ form of the type $\xi=a\wedge g$ where $0\ne a\in\Omega^{(1,0)}(X_{p})$ and $0\ne g\in End(E_{p})$, we have 
\[\mathrm{Tr}\left(\sum_{j=0}^{n-1}\sqrt{-1}\xi\wedge(\sqrt{-1}F_{H}(p))^{j}\wedge \xi^{*}\wedge(\sqrt{-1}F_{H}(p))^{n-1-j}\right)>0.\] 
 We take the curvature at a point as above and $a=dz^{2}$ and $g=E_{11}$.
	Now let us calculate the condition $E$ for the vbMA equation. 
	\begin{align*}
		&\mathrm{Tr}\left(\sum_{j=0}^{2}\sqrt{-1}ag\wedge(\sqrt{-1}F_{H})^{j}\wedge \bar{a}g^{*}\wedge(\sqrt{-1}F_{H})^{2-j}\right)\\
		&=\mathrm{Tr}\left(\sqrt{-1}a\wedge\bar{a}gg^{*}\wedge(\sqrt{-1}F_{H})^{2}+\sqrt{-1}a\wedge\bar{a}g^{*}g\wedge(\sqrt{-1}F_{H})^{2}\right)\\
		&+\mathrm{Tr}\left(\sqrt{-1}a\wedge \bar{a}g\wedge(\sqrt{-1}F_{H})\wedge g^{*}\wedge\sqrt{-1}F_{H}\right).
	\end{align*}
Since $g=E_{11}$, we have 
\begin{align*}
	&=2(2px-b^{2})+2px\\
	&=2(3px-b^{2})
\end{align*}

	This explains the $3px-b^{2}=0$ condition at $t=1$. Now we prove theorem \ref{thmvbma}.
\begin{thm}[$=$ Theorem \ref{thmvbma}]
	The ellipticity is not preserved for the vector bundle Monge-Amp\`ere equation at a point along continuity paths in the connected component of trivial solution for dimension $\ge 3$ and rank $\ge 3$.
\end{thm}
\begin{proof}
	Our previous counter example is the case dimension $=3$ and rank $=3$. We will generalise this. So previously we had 
	\[(\sqrt{-1}F_{H})^{3}=K(t)\mathrm{Id}.\] Observe that we can divide all the matrices by $K(t)^{\frac{1}{3}}$ and get 
	\[(\sqrt{-1}F_{H})^{3}=\mathrm{Id}.\] Here we are misusing notation and writing $\sqrt{-1}F_{H}$(even after divison by $K(t)^{\frac{1}{3}}$). First, we generalise to the case dimension $=3$ and rank $\ge 3$. This is easily achieved by exending the matrices $A_{i}$ and $B_{ij}$ in the following manner.
	\[A_{i}'=\begin{bmatrix}
		A_{i} & 0\\
		0 & \frac{1}{6^{\frac{1}{3}}}\mathrm{Id}
	\end{bmatrix}, B_{ij}'=\begin{bmatrix}
	B_{ij} & 0\\
	0 & 0
	\end{bmatrix},\]
	where the matrices are written in a block diagonal way. Now we are ready to generalise this to higher dimension. For this we write 
		\begin{align*}
			&\sqrt{-1}F_{H}\\
			&=\sum_{i=1}^{3}A_{i}'\sqrt{-1}dz^{i}\wedge d\bar{z}^{i}+\sum_{1\le i<j\le3}B_{ij}'\sqrt{-1}dz^{i}\wedge d\bar{z}^{j}+\sum_{1\le i<j\le3}B_{ij}'^{*}\sqrt{-1}dz^{j}\wedge d\bar{z}^{i}+\sum_{l=1}^{k}\mathrm{Id}\sqrt{-1}dz^{l+3}\wedge d\bar{z}^{l+3}\\
			&=\mathcal{A}+\sum_{l=1}^{k}\mathrm{Id}\sqrt{-1}dz^{l+3}\wedge d\bar{z}^{l+3}
				\end{align*}
		 Observe that 
			\[(\sqrt{-1}F_{H})^{k+3}=\binom{k+3}{3}\mathcal{A}^{3}(\sum_{l=1}^{k}\sqrt{-1}dz^{l+3}\wedge d\bar{z}^{l+3})^{k}=\binom{k+3}{3}k!\mathrm{Id}\]
			for any $k$. Hence vector bundle Monge-Amp\`ere equation is satisfied in any dimension (albeit the constant $\binom{k+3}{3}k!$ will vary but can be divided). Now we only need to show that the condition $E$ also stays the same. For this we take $\xi=dz^{2}E_{11}$. Now
			\begin{align*}
				&\sum_{j=0}^{k+2}\mathrm{Tr}\left(\sqrt{-1}\xi\wedge (\sqrt{-1}F_{H})^{j}\wedge \xi^{*}\wedge(\sqrt{-1}F_{H})^{k+2-j}\right)\\
				&=\sum_{j=0}^{k+2}\mathrm{Tr}\left(E_{11} (\sqrt{-1}F_{H})^{j}\wedge E_{11}\wedge(\sqrt{-1}F_{H})^{k+2-j}\right)\wedge\sqrt{-1}dz^{2}\wedge d\bar{z}^{2}.
			\end{align*}
		This can be seen as the coefficient of $m$  in
			\begin{align*}	
			&\mathrm{Tr}\left(E_{11}\big(\sqrt{-1}F_{H}+mE_{11}\sqrt{-1}dz^{2}\wedge d\bar{z}^{2}\big)^{k+3}\right)\\
			&=\mathrm{Tr}\big(E_{11}(\sqrt{-1}F_{H})^{k+3}\big)+m\mathrm{Tr}\left(E_{11}\sum_{j=0}^{k+2}(\sqrt{-1}F_{H})^{j}E_{11}(\sqrt{-1}F_{H})^{k+2-j}\right)\wedge\sqrt{-1}dz^{2}\wedge d\bar{z}^{2}.
		\end{align*}
		Now let us compute the coefficient of $m$:
		\begin{align*}
			&\mathrm{Tr}\left(E_{11}\big(\sqrt{-1}F_{H}+mE_{11}\sqrt{-1}dz^{2}\wedge d\bar{z}^{2}\big)^{k+3}\right)\\
			&=\mathrm{Tr}\left(E_{11}\big((\mathcal{A}+mE_{11}\sqrt{-1}dz^{2}\wedge d\bar{z}^{2})+\sum_{l=1}^{k}\mathrm{Id}\sqrt{-1}dz^{l+3}\wedge d\bar{z}^{l+3}\big)^{k+3}\right)\\
			&=\mathrm{Tr}\left(\binom{k+3}{3}E_{11}\big(\mathcal{A}+mE_{11}\sqrt{-1}dz^{2}\wedge d\bar{z}^{2}\big)^{3}\big(\sum_{l=1}^{k}\sqrt{-1}dz^{l+3}\wedge d\bar{z}^{l+3}\big)^{k}\right).
		\end{align*}
		The coefficient of $m$  in $\mathrm{Tr}\left(E_{11}\big(\mathcal{A}+mE_{11}\sqrt{-1}dz^{2}\wedge d\bar{z}^{2}\big)^{3}\right)$ is $2(3px-b^{2})$. Indeed 
		\begin{align*}
			&\mathrm{Tr}\left(E_{11}\big(\mathcal{A}+mE_{11}\sqrt{-1}dz^{2}\wedge d\bar{z}^{2}\big)^{3}\right)\\
			&=\mathrm{Tr}\big(E_{11}\mathcal{A}^{3}\big)+m \mathrm{Tr}\left(E_{11}\sum_{j=0}^{2}\mathcal{A}^{j}E_{11}\mathcal{A}^{2-j}\right)\\
			&=\mathrm{Tr}\big(E_{11}\mathcal{A}^{3}\big)+m\sum_{j=0}^{2} \mathrm{Tr}\left(E_{11}\mathcal{A}^{j}E_{11}\mathcal{A}^{2-j}\right).
		\end{align*}	
		Hence the condition $E$ for the vbMA equation becomes $\frac{(k+3)!}{3}(3px-b^{2})$.

\end{proof}

\section{Counter Example For the $\sigma_{k}$ Equations when $k\ge3$}
\label{sectionsigmak}
In this section we will show that the ellipticity is not preserved for $\sigma_{k}$ equations when $k\ge 3$. We will use the counter example for vector bundle Monge-Amp\`ere equation in dimension $k$ and rank $3$. So we have $\sqrt{-1}F_{H}=\sum_{i=1}^{k}A_{i}\sqrt{-1}dz^{i}\wedge d\bar{z}^{i}+\sum_{1=i<j=k}B_{ij}\sqrt{-1}dz^{i}\wedge d\bar{z}^{j}+\sum_{1=i<j=k}B_{ij}^{*}\sqrt{-1}dz^{j}\wedge d\bar{z}^{i}$ and $\omega=\sum_{i=1}^{k}\sqrt{-1}dz^{i}\wedge d\bar{z}^{i}$. They satisfy $(\sqrt{-1}F_{H})^{k}=\mathrm{Id}$ and  $3p(1)x(1)-b(1)^{2}=0$. Here we have ignored the volume form on the right.\\ Now the $\sigma_{k}$ equations for $k\ge 3$ is given by 
\[(\sqrt{-1}F_{H})^{k}\wedge\omega^{n-k}=\eta \mathrm{Id}.\]
The $\sigma_{k}$ equation in dimension $k$ is same as the vector bundle Monge-Amp\`ere equation. So we will assume $n\ge k+1$. Now let us take
\[\sqrt{-1}F_{H}=\sum_{i=1}^{k}A_{i}\sqrt{-1}dz^{i}\wedge d\bar{z}^{i}+\sum_{1\le i<j\le k}B_{ij}\sqrt{-1}dz^{i}\wedge d\bar{z}^{j}+\sum_{1\le i<j\le k}B_{ij}^{*}\sqrt{-1}dz^{j}\wedge d\bar{z}^{i}+\epsilon \sum_{l=1}^{n-k}\mathrm{Id}\sqrt{-1}dz^{l+k}\wedge d\bar{z}^{l+k}\]
and 
\[\omega=\epsilon\sum_{i=1}^{k}\sqrt{-1}dz^{i}\wedge d\bar{z}^{i}+\sum_{l=1}^{n-k}\sqrt{-1}dz^{l+k}\wedge d\bar{z}^{l+k},\]
where $\epsilon>0$ is small and to be chosen later. \\
First let us  calculate the  condition $E$ for $\sigma_{k}$ equation. For this we take $\xi=dz^{2}E_{11}$. Now
\[\sum_{j=0}^{k-1}\mathrm{Tr}\left(\sqrt{-1}\xi\wedge(\sqrt{-1}F_{H})^{j}\wedge\xi^{*}\wedge(\sqrt{-1}F_{H})^{k-1-j}\right)\wedge\omega^{n-k}\] 
becomes 
\[\frac{(k+3)!}{3}(n-k)!(3px-b^{2})+b^{2}G(\epsilon)+H(p,x,\epsilon),\]	where for $\epsilon=0$, $G(\epsilon=0)=0=H(p,x,\epsilon=0)$. Misusing notation we write 
\[\frac{(k+3)!(n-k)!}{3}\left((3px-b^{2})+b^{2}G(\epsilon)+H(p,x,\epsilon)\right).\]
As before we choose $p=x=1$ for all $t\in [0,1]$ and $b=\sqrt{
	\frac{(3+H(1,1,\epsilon))t}{1-G(\epsilon)}}$ for small $\epsilon$. This ensures that the condition $E$ for $\sigma_{k}$ is violated at $t=1$. Next we will show that we can solve the $\sigma_{k}$ equation. 	Now for this we write 
\[\sqrt{-1}F_{H}=\mathcal{A}+\epsilon\sum_{l=1}^{n-k}\mathrm{Id}\sqrt{-1}dz^{l+k}\wedge d\bar{z}^{l+k}\]
and 
\[\omega=\epsilon\sum_{i=1}^{k}\sqrt{-1}dz^{i}\wedge d\bar{z}^{i}+\mathcal{B}.\]	
We calculate 
\begin{align*}
	&(\sqrt{-1}F_{H})^{k}\wedge\omega^{n-k}\\
	&=(\mathcal{A}+\epsilon\sum_{l=1}^{n-k}\mathrm{Id}\sqrt{-1}dz^{l+k}\wedge d\bar{z}^{l+k})^{k}\wedge(\epsilon\sum_{i=1}^{k}\sqrt{-1}dz^{i}\wedge d\bar{z}^{i}+\mathcal{B})^{n-k}\\
	&=\mathcal{A}^{k}\wedge\mathcal{B}^{n-k}+\mathcal{D}_{\epsilon},
\end{align*}
where $\mathcal{D}_{\epsilon}$ is a diagonal matrix with $\mathcal{D}_{\epsilon,11}$ is a function of $p,s,x,v,y,a,b,c,\epsilon$; $\mathcal{D}_{\epsilon,22}$ is a function of $q,u,y,r,x,a,b,c,\epsilon$; $\mathcal{D}_{\epsilon,33}$ is a function of $r,v,z,q,s,a,b,c,\epsilon$ and for $\epsilon=0$, $\mathcal{D}_{\epsilon}$ is the zero matrix. Now the above will give the following three expressions
\begin{align*}
	&1)	\alpha:=\binom{k}{3}(k-3)!(n-k)!\left(6psx-b^{2}(2s+v)-a^{2}(2x+y)+2abc\right)+D_{\epsilon,11}\\
	&2) \beta:=\binom{k}{3}(k-3)!(n-k)!\left(6quy-c^{2}(2q+r)-a^{2}(x+2y)+2abc\right)+D_{\epsilon,22}\\
	&3) \gamma:=\binom{k}{3}(k-3)!(n-k)!\left(6rvz-c^{2}(q+2r)-b^{2}(s+2v)+2abc\right)+D_{\epsilon,33}.
\end{align*}	
 We have already defined $p,x,b$. To set up for the implicit function theorem, we need to define seven more functions as the above three expressions will only give two independent equations. So we define 
\begin{align*}
	&s,y=1\\
	& q=1+t,r=1+2t,v=1-0.9t\\
	&c=a=\sqrt{
		\frac{(3+H(1,1,\epsilon))t}{1-G(\epsilon)}}.
\end{align*} 
Now let us define 
\[T:\mathbb{R}\times\mathbb{R}^{2}\rightarrow\mathbb{R}^{2}\]
\[(\epsilon,u,z)=(\beta-\alpha,\gamma-\alpha).\]
Now we calculate at $\epsilon=0$
\[\begin{bmatrix}
	\frac{\partial(\beta-\alpha)}{\partial u} & \frac{\partial(\beta-\alpha)}{\partial z}\\
	\frac{\partial(\gamma-\alpha)}{\partial u} & \frac{\partial(\gamma-\alpha)}{\partial z}\\
\end{bmatrix}=\binom{k}{3}(k-3)!(n-k)!\begin{bmatrix}
	6qy & 0\\
	0 & 6rv
\end{bmatrix}.\]
The determinant is $\{\binom{k}{3}(k-3)!(n-k)!\}^{2}36qyrv=\{\binom{k}{3}(k-3)!(n-k)!\}^{2}36(1+t)\times 1\times (1+2t)\times (1-0.9t)$ which is non-zero for all $t\in[0,1]$. So we have solution for small $\epsilon$ and this will only give $\alpha=\beta=\gamma$. We can make sure that $\alpha=\beta=\gamma>0$ by choosing smaller $\epsilon$ if needed since at $\epsilon=0$ we have positivity. Hence we are done. This completes the rank three case in every possible dimensions. The higher rank is an easy extension of this like the vector bundle Monge-Amp\`ere equation. This proves the first part of theorem \ref{thmsigma}.

\section{Ellipticity Preservation of the $\sigma_{2}$ Equation}	
\label{sectionsigma2}		
In this section, we prove that the equation 
\[(\sqrt{-1}F_{H})^{2}\wedge \omega^{n-2}=\eta\mathrm{Id}\]
preserves ellipticity in any dimension. 	\\		
Now at a point, the curvature of a metric is of the form
\[\sqrt{-1}F_{H_{0}}=\sum_{i=1}^{n}A_{i}\sqrt{-1}dz^{i}\wedge d\bar{z}^{i}+\sum_{i<j}B_{ij}\sqrt{-1}dz^{i}\wedge d\bar{z}^{j}+\sum_{i<j}B_{ij}^{*}\sqrt{-1}dz^{j}\wedge d\bar{z}^{i},\]	
 the K\"ahler form
\[\omega=\sum_{i=1}^{n}\sqrt{-1}dz^{i}\wedge d\bar{z}^{i},\]
and the volume form 
\[\eta=C\sqrt{-1}dz^{1}\wedge d\bar{z}^{1}\dots\sqrt{-1}dz^{n}\wedge d\bar{z}^{n}\]
for some positive constant $C$.	
Then after routine calculations, we get	(there is a positive constant but we are ignoring it)
\[\frac{(\sqrt{-1}F_{H_{0}})^{2}\wedge\omega^{n-2}}{\eta}=\sum_{i<j} (A_{i}A_{j}+A_{j}A_{i}- B_{ij}B_{ij}^{*}-B_{ij}^{*}B_{ij})\]	
Now suppose $\sqrt{-1}F_{H_{0}}$ satisfies the $\sigma_{2}$ equation. Then we have 
\begin{align*}
	&\sum_{i<j} (A_{i}A_{j}+A_{j}A_{i}- B_{ij}B_{ij}^{*}-B_{ij}^{*}B_{ij})=\mathrm{Id}>0\\
	&\implies\sum_{i<j} (A_{i}A_{j}+A_{j}A_{i})>\sum_{i<j} (B_{ij}B_{ij}^{*}+B_{ij}^{*}B_{ij})\\
	&\implies\sum_{i<j}(A_{i}A_{j}+A_{j}A_{i})>0
\end{align*}	
The last line follows because for any $B_{ij}$, we know $(B_{ij}B_{ij}^{*}+B_{ij}^{*}B_{ij})\ge 0$. The inequalities are in terms of positive definiteness or positive semi-definiteness.	\\	
The condition $E$ for $\sigma_{2}$ equation is that for all non-zero $\xi$, we have 
\[\mathrm{Tr}\left(\sum_{j=0}^{1}\sqrt{-1}\xi\wedge (\sqrt{-1}F_{H})^{j}\wedge \xi^{*}\wedge(\sqrt{-1}F_{H})^{1-j}\wedge\omega^{n-2}\right)>0.\]
Without loss of generality we can choose $\xi=dz^{k}g$, then this condition translates to(ignoring the volume form part)
\[Tr\left((gg^{*}+g^{*}g)\sum_{i\ne k}A_{i}\right)>0\]	
\begin{lem}
Condition $E$ for the $\sigma_{2}$ equation with $\xi=dz^{k}g$ $\iff$ $\sum_{i\ne k}A_{i}$ is positive definite.
\end{lem}
It is easy to see from the above discussion 	that if $\sum_{i\ne k}A_{i}$  is positive definite for all $k$, then Condition $E$ for the $\sigma_{2}$-equation holds. Note that preservation of this condition would imply condition $E$ is preserved. Now we prove the reverse direction. Suppose condition $E$ holds then for any $k$ and non-zero matrix $g$ we have 
\[\mathrm{Tr}\left((gg^{*}+g^{*}g)\sum_{i\ne k}A_{i}\right)>0.\]
In particular for any positive semidefinite matrix $H$, we take $g=\frac{1}{\sqrt{2}}H^{\frac{1}{2}}=g^{*}$. Then the above gives for any positive semidefinite matrix $H$, we have 
\[\mathrm{Tr}\left(H(\sum_{i\ne k}A_{i})\right)=\mathrm{Tr}\left((\sum_{i\ne k}A_{i})H\right)>0.\]
Now for any non-zero vector $x$, we take the positive semidefinite matrix $H=xx^{*}$. So we have 
\begin{align*}
	& \mathrm{Tr}\left((\sum_{i\ne k}A_{i})xx^{*}\right)>0\\
	&\implies\mathrm{Tr}\left(x^{*}(\sum_{i\ne k}A_{i})x\right)>0
\end{align*}
This gives the desired result.
\begin{lem}
	Suppose $H_{t}$ is a path of metrics such that $H_{0}$ satisfies the condition : $\sum_{i\ne k}A_{i,t=0}$ is positive definite for each $k$ and $\frac{(\sqrt{-1}F_{H_{t}})^{2}\wedge \omega^{n-2}}{\eta}$ is positive defiite for all $t$ then $\sum_{i\ne k}A_{i,t}$ is positive definite for each $k$. 
\end{lem}
\begin{proof}
	Suppose $t_{0}$ is the first time where the condition is violated. Then we have $\sum_{i\ne k}A_{i,t_{0}}$ is not positive definite for some $k$. We will suppress the $t_{0}$ from hereon. By continuity it will become positive semi-definite. Hence there exists a non-zero vector $v$ such that 
	\[v^{*}(\sum_{i\ne k}A_{i})v=0.\]
	This gives (positive semi-definiteness)
	\[(\sum_{i\ne k}A_{i})v=0.\]
	Multiplying this equation from the left by $A_{k}$, we have 
	\[(\sum_{i\ne k}A_{k}A_{i})v=0.\] 
	Taking $()^{*}$, we get
	\[v^{*}(\sum_{i\ne k}A_{i}A_{k})=0.\]
	But since at  this point $\frac{(\sqrt{-1}F_{H_{t}})^{2}\wedge \omega^{n-2}}{\eta}$ is positive definite, we have 
	\begin{align*}
		& v^{*}(\sum_{i\ne j}(A_{i}A_{j}+A_{j}A_{i}))v>0\\
		&\implies  v^{*}(\sum_{i\ne j; i,j\ne k}(A_{i}A_{j}+A_{j}A_{i}))v>0.
	\end{align*} 
	In the last line we have used $(\sum_{i\ne k}A_{k}A_{i})v=0=v^{*}(\sum_{i\ne k}A_{i}A_{k})$. Now suppose $l=\min\{1,\dots,k-1,k+1,\dots,n \}$. Now the equation $(\sum_{i\ne k}A_{i})v=0$ can be written as $A_{l}v=-\sum_{i\ne k,l}A_{i}v$.
	Now we get
	\begin{align*}
		& v^{*}(\sum_{i\ne j; i,j\ne k}(A_{i}A_{j}+A_{j}A_{i}))v>0\\
		&\implies  v^{*}(\sum_{j\ne k,l}(A_{l}A_{j}+A_{j}A_{l}))v+ v^{*}(\sum_{i\ne j; i,j\ne k,l}(A_{i}A_{j}+A_{j}A_{i}))v>0\\
		&\implies  -2v^{*}(\sum_{i\ne k,l}A_{i})^{2}v+ v^{*}(\sum_{i\ne j; i,j\ne k,l}(A_{i}A_{j}+A_{j}A_{i}))v>0\\
		&\implies-v^{*}(\sum_{i\ne k,l}A_{i})^{2}v-v^{*}(\sum_{i\ne k,l}A_{i}^{2})v>0\\
		&\implies -\|(\sum_{i\ne k,l}A_{i})v\|^{2}-\sum_{i\ne k,l}\|A_{i}v\|^{2}>0
	\end{align*}
	This is a contradiction. Thus $\sum_{i\ne k}A_{i,t}$ is positive definite for each $k$ throughout the path. Hence condition $E$ is preserved.
\end{proof}
The above lemma proves the second part of therem \ref{thmsigma}.

		\section{Counter Example for the Vector Bundle J-equation}
		\label{section J-equation}
		The vector bundle $J$-equation was introduced by Takahashi (\cite{Takahashi2024}).  Suppose $E$ be a
		holomorphic vector bundle of rank $r$ over an $n$-dimensional compact Kähler manifold $(X,\omega)$
		satisfying
	\[	ch_{n}(E)>0,\ \ \  [\omega]\cdot ch_{n-1}(E) > 0.\]
		 We say that a Hermitian
		metric $H$  solves the $J$-equation if it satisfies
		\[c(\sqrt{-1}F_{H})^{n}-\omega\mathrm{Id}\wedge (\sqrt{-1}F_{H})^{n-1}=0,\]
		for some constant $c$.  In fact 
		\[c=\frac{[\omega]\cdot ch_{n-1}(E)}{n ch_{n}(E)}>0.\]
	Dividing by $c$ and considering $\frac{\omega}{c}$ a K\"ahler form, we will consider the following equation
	\[(\sqrt{-1}F_{H})^{n}-\omega\mathrm{Id}\wedge(\sqrt{-1}F_{H})^{n-1}=0\] as the vector bundle $J$-equation.\\	We will be working in dimension $3$ first. So at a point, the curvature takes the form
		\[\sqrt{-1}F_{H_{0}}=\sum_{i=1}^{3}A_{i}\sqrt{-1}dz^{i}\wedge d\bar{z}^{i}+\sum_{i<j}(B_{ij}\sqrt{-1}dz^{i}\wedge d\bar{z}^{j}+B_{ij}^{*}dz^{j}\wedge d\bar{z}^{i}),\]
		then we have(ignoring the volume form part)
		\begin{align*}
			&(\sqrt{-1}F_{H_{0}})^{3}-\omega\mathrm{Id}\wedge (\sqrt{-1}F_{H_{0}})^{2}\\
			&=\{A_{1},A_{2},A_{3}\}-\{A_{1},B_{23},B_{23}^{*}\}-\{A_{2},B_{13},B_{13}^{*}\}-\{A_{3},B_{12},B_{12}^{*}\}\\
			&+\{B_{13},B_{12}^{*},B_{23}^{*}\}+\{B_{13}^{*},B_{12},B_{23}\}-\{A_{1},A_{2}\}-\{A_{1},A_{3}\}-\{A_{2},A_{3}\}\\
			&+\{B_{12},B_{12}^{*}\}+\{B_{13},B_{13}^{*}\}+\{B_{23},B_{23}^{*}\},
		\end{align*}
		where $\{X,Y\}=XY+YX$.
The above is a routine computation.\\
		 We consider the following matrices 
		\[A_{1}=\begin{bmatrix}
			p(t) & 0 & 0\\
			0 & q(t) & 0\\
			0 & 0  & r(t)
		\end{bmatrix}, \ \ \  A_{2}=\begin{bmatrix}
		s(t) & 0 & 0\\
		0 & u(t) & 0\\
		0 & 0 & v(t)
		\end{bmatrix}, \ \ \ 
		A_{3}=\begin{bmatrix}
			x(t) & 0 & 0\\
			0 & y(t) & 0\\
			0 & 0 & z(t)
		\end{bmatrix};\] and  
		 	\[B_{12}=\begin{bmatrix}
		0 & a(t) & 0  \\
		0 & 0 & 0\\
		0 & 0 & 0
		\end{bmatrix}, \ \ \  B_{13}=\begin{bmatrix}
		0 & 0 & b(t)\\
		0 & 0 & 0\\
		0 & 0 & 0 
		\end{bmatrix}, \ \ \ 
		B_{23}=\begin{bmatrix}
		0 & 0 & 0\\
		0 & 0 & c(t)\\
		0 & 0 & 0
		\end{bmatrix}.\]
		 The matrices $\{A_{1},A_{2},A_{3}\},\{A_{1},B_{23},B_{23}^{*}\},\{A_{2},B_{13},B_{13}^{*}\},$
		$\{A_{3},B_{12},B_{12}^{*}\},\{B_{13}^{*},B_{12},B_{13}\},\\
		\{B_{13},
		B_{12}^{*},B_{23}^{*}\}$ are same as before in section \ref{section vbma}. The computations of $\{A_{i},A_{j}\}$(for $i\ne j$) and $\{B_{ij},B_{ij}^{*}\}$ are easy to see.\\		
	With this considerations, the vector bundle $J$-equation takes the form
	\begin{equation}
	\label{Jequation}
	\begin{cases}
		1)6psx-b^{2}(2s+v-1)-a^{2}(2x+y-1)+2abc-2(ps+sx+px)=0\\
		 2)6quy-c^{2}(2q+r-1)-a^{2}(x+2y-1)+2abc-2(qu+uy+qy)=0\\
		 3)6rvz-c^{2}(q+2r-1)-b^{2}(s+2v-1)+2abc-2(rv+vz+rz)=0.
	\end{cases}
		\end{equation}	

The condition $E$ for the vector bundle $J$-equation is that $3ps-a^{2}-p-s=0$. \\
Our initial point will be $p=q=r=s=u=v=2$, $x=y=z=\frac{1}{2}$, and $a=b=c=0$. We define $p,s,x,y,a:[0,1]\rightarrow \mathbb{R}$ by 
\begin{align*}
	& 1)p(t)=2+2t(1-t) \\
	& 2) s(t)=2-t\\
	& 3) x(t)=\frac{1}{2}\\
	& 4) y(t)=\frac{1}{2}\\
	& 5) a(t)=\sqrt{3}t. 
\end{align*}
 Observe that the condition $E$ for the vector bundle $J$-equation is not satisfied at $t=1$ and $p(0)=s(0)=2,x(0)=y(0)=\frac{1}{2},a(0)=0$. Then the first equation in (\ref{Jequation}) becomes 
\begin{align*}
	&6psx-b^{2}(2s+v-1)-a^{2}(2x+y-1)+2abc-2(ps+sx+px)=0\\
	& \implies x(6ps-2a^{2}-2s-2p)=b^{2}(2s+v-1)+a^{2}(y-1)+2ps-2abc\\
	&\implies 3ps-a^{2}-s-p=b^{2}(2s+v-1)+a^{2}y-a^{2}+2ps-2abc\\
	&\implies ps-p-s=b^{2}(2s+v-1)+a^{2}y-2abc.
\end{align*}
In the third line, we have used $x=\frac{1}{2}$. Now putting the functions $p,s,y,a$, we get
\begin{align*}
	&(2+2t-2t^{2})(2-t)-2-2t+2t^{2}-2+t=b^{2}(4-2t+v-1)+\frac{3}{2}t^{2}-2\sqrt{3}tbc\\
	&\implies (t-4t^{2}+2t^{3})=b^{2}(3-2t+v)+\frac{3}{2}t^{2}-2\sqrt{3}tbc
\end{align*}
Now we define $b(t)=\sqrt{\frac{t}{5}}$ and putting this in the above, we get
\begin{align*}
	&(t-4t^{2}+2t^{3})=\frac{t}{5}(3-2t+v)+\frac{3}{2}t^{2}-\sqrt{\frac{12}{5}}t\sqrt{t}c\\
	&\implies(1-4t+2t^{2})=\frac{3-2t+v}{5}+\frac{3}{2}t-\sqrt{\frac{12t}{5}}c.
\end{align*} 
Next, we define $c(t)=5.5\times\sqrt{\frac{5t}{12}}$ and putting this above, we get
\begin{align*}
	&1-4t+2t^{2}=\frac{3-2t+v}{5}+1.5t-5.5t\\
	&\implies v=2+2t+10t^{2}.
\end{align*}
So we define $v(t)$ by the above formula. Hence we have solved equation $(1)$ in (\ref{Jequation}) for all $t\in[0,1]$. Next we solve Equation $(2)$ and $(3)$ in (\ref{Jequation}).\\
For this, we define $q(t)=r(t)=2$. Now from second equation in (\ref{Jequation}) we get
\begin{align*}
	&6quy-c^{2}(2q+r-1)-a^{2}(x+2y-1)+2abc-2(qu+uy+qy)=0\\
	& \implies 6u-5c^{2}-\frac{a^{2}}{2}+2abc-2(1+2.5u)=0\\
	&\implies u=2+5c^{2}+\frac{a^{2}}{2}-2abc.
\end{align*}
Now putting $a,b,c$ we get
\begin{align*}
	& u=2+5(5.5)^{2}\times \frac{5t}{12}+\frac{3t^{2}}{2}-2\times (5.5)\sqrt{\frac{3\times 5}{5\times 12}}t^{2}\\
	& \implies u=2+\frac{25\times (5.5)^{2}}{12}t-4t^{2}.
\end{align*}
We define $u(t)$ by the above formula. Now equation$(2)$ in (\ref{Jequation}) is satisfied. \\
Next we consider equation $(3)$ in (\ref{Jequation}).
\begin{align*}
	&6rvz-c^{2}(q+2r-1)-b^{2}(s+2v-1)+2abc-2(rv+vz+rz)=0\\
	&\implies 12v\times z-5c^{2}-b^{2}(2-t+4+4t+20t^{2}-1)+5.5t^{2}-4v-2z(r+v)=0\\
	&\implies z=\frac{8+(9+\frac{25\times (5.5)^{2}}{12})t+(\frac{3}{5}+34.5)t^{2}+4t^{3}}{16+20t+100t^{2}}.
\end{align*}
 If we define $z$ as above then the third equation in (\ref{Jequation}) is satisfied.\\
Now let us show that the condition $E$ for the vector bundle $J$-equation does not hold at $t=1$. Indeed, we consider the following test form $\xi=dz^{3}g=dz^{3}\begin{bmatrix}
	1 & 0 & 0\\
	0 & 0 & 0\\
	0 & 0 & 0
\end{bmatrix}$. Now
\begin{align*}
	&\mathrm{Tr}\left(\sum_{j=0}^{2}\sqrt{-1}\xi\wedge(\sqrt{-1}F_{H})^{j}\wedge \xi^{\dagger}\wedge (\sqrt{-1}F_{H})^{2-j}\right)-\mathrm{Tr}\left(\sum_{j=0}^{1}\sqrt{-1}\xi\wedge(\sqrt{-1}F_{H})^{j}\wedge\xi^{\dagger}\wedge(\sqrt{-1}F_{H})^{1-j}\right)\wedge \omega\\
	&=\mathrm{Tr}\left((gg^{*}+g^{*}g)(\{A_{1},A_{2}\}-\{B_{12},B_{12}^{*}\})+(gA_{1}g^{*}A_{2}+gA_{2}g^{*}A_{1}-gB_{12}g^{*}B_{12}^{*}-gB_{12}^{*}g^{*}B_{12})\right)\\
	&-\mathrm{Tr}\left((gg^{*}+g^{*}g)(A_{1}+A_{2})\right)\\
	&=2(2ps-a^{2}-p-s)+ps+ps-0-0\\
	&=2(3ps-a^{2}-p-s)
\end{align*}
Now when $a=b=c=0$, all the equations become of the form
\[6rvz-2(rv+vz+rz)=0.\]
Now if we divide by $2rvz$(we are considering $r,v,,z>0$) then we get
\[3=\frac{1}{r}+\frac{1}{v}+\frac{1}{z}.\]
We take $\frac{1}{r}=\frac{1}{1+t}, \frac{1}{v}=\frac{1}{1+t}$. Then $\frac{1}{z}=3-\frac{2}{1+t}=\frac{1+3t}{1+t}$ and hence $z(t)=\frac{1+t}{1+3t}$. If we define $p=q=s=u=r=v$, then $x=y=z$ and they satisfy the equations for $a=b=c=0$. This is a path solving the equations and from the point $p=q=r=s=u=v=x=y=x=1$ to $p=q=r=s=u=v=2; x=y=z=\frac{1}{2}$. Concatenating these two paths, we get our desired path. Also one can show condition $E$ for the vector bundle $J$-equation holds at $p=q=r=s=u=v=2;x=y=z=\frac{1}{2}$. This is done in section \ref{section dhym}.
\begin{thm}[$=$Theorem \ref{thmj}]
	The ellipticity is not preserved for the vector bundle $J$-equation at a point along continuity paths in the connected component of trivial solution when dimension $\ge 3$ and rank $\ge 3$.
\end{thm}
\begin{proof}
	The previous example proves the theorem when dimension $=3$ and rank $=3$. First we extend to the case dimension $=3$ and rank $\ge 3$. This is done in the following way. We define 
	\[A_{i}'=\begin{bmatrix}
		A_{i} & 0\\
		0 & \mathrm{Id}
	\end{bmatrix}, B_{ij}'=\begin{bmatrix}
	 B_{ij} & 0\\
	 0 & 0
	\end{bmatrix},\]
	where the matrices are written in a block diagonal way. Now we will extend this to higher dimensions.  Hence we write
\begin{align*}
	&\sqrt{-1}F_{H}\\
	&=\sum_{i=1}^{3}A_{i}'\sqrt{-1}dz^{i}\wedge d\bar{z}^{i}+\sum_{1\le i<j\le3} B_{ij}'\sqrt{-1}dz^{i}\wedge d\bar{z}^{j}+\sum_{1\le i<j\le3} B_{ij}'^{*}\sqrt{-1}dz^{j}\wedge d\bar{z}^{i}+\sum_{l=1}^{k}\mathrm{Id}\sqrt{-1}dz^{l+3}\wedge d\bar{z}^{l+3}\\
	&=\mathcal{A}+\sum_{l=1}^{k}\mathrm{Id}\sqrt{-1}dz^{l+3}\wedge d\bar{z}^{l+3}
\end{align*}
and 
\[\omega=\sum_{i=1}^{k+3}\sqrt{-1}dz^{i}\wedge d\bar{z}^{i}.\]
Now we calculate 
\begin{align*}
	&(\sqrt{-1}F_{H})^{k+3}-\omega\wedge (\sqrt{-1}F_{H})^{k+2}\\
	&=\binom{k+3}{3}\mathcal{A}^{3}\wedge(\sum_{l=1}^{k}\mathrm{Id}\sqrt{-1}dz^{l+3}\wedge d\bar{z}^{l+3})^{k}\\
	&-\omega\wedge \left(\binom{k+2}{2}\mathcal{A}^{2}(\sum_{l=1}^{k}\mathrm{Id}\sqrt{-1}dz^{l+3}\wedge d\bar{z}^{l+3})^{k}+\binom{k+2}{3}\mathcal{A}^{3}\wedge (\sum_{l=1}^{k}\mathrm{Id}\sqrt{-1}dz^{l+3}\wedge d\bar{z}^{l+3})^{k-1}\right)\\
	&=\binom{k+3}{3}\mathcal{A}^{3}\wedge(\sum_{l=1}^{k}\mathrm{Id}\sqrt{-1}dz^{l+3}\wedge d\bar{z}^{l+3})^{k}-\binom{k+2}{3}\mathcal{A}^{3}\wedge (\sum_{l=1}^{k}\mathrm{Id}\sqrt{-1}dz^{l+3}\wedge d\bar{z}^{l+3})^{k}\\
&-\binom{k+2}{2}\mathcal{A}^{2}\wedge(\sum_{i=1}^{3}\sqrt{-1}dz^{i}\wedge d\bar{z}^{i})\wedge(\sum_{l=1}^{k}\mathrm{Id}\sqrt{-1}dz^{l+3}\wedge d\bar{z}^{l+3})^{k}\\
&=\binom{k+2}{2}\mathcal{A}^{3}\wedge (\sum_{l=1}^{k}\mathrm{Id}\sqrt{-1}dz^{l+3}\wedge d\bar{z}^{l+3})^{k}\\
&-\binom{k+2}{2}\mathcal{A}^{2}\wedge(\sum_{i=1}^{3}\sqrt{-1}dz^{i}\wedge d\bar{z}^{i})\wedge(\sum_{l=1}^{k}\mathrm{Id}\sqrt{-1}dz^{l+3}\wedge d\bar{z}^{l+3})^{k}\\
&=\binom{k+2}{2}(\sum_{l=1}^{k}\mathrm{Id}\sqrt{-1}dz^{l+3}\wedge d\bar{z}^{l+3})^{k}\left(\mathcal{A}^{3}-\mathcal{A}^{2}\wedge (\sum_{i=1}^{3}\sqrt{-1}dz^{i}\wedge d\bar{z}^{i}) \right)\\
&=0.
\end{align*}
Now we only need to check the condition $E$ for the vector bundle $J$-equation. For this we take $\xi=dz^{3}E_{11}$. So the condition $E$ for the vector bundle $J$-equation is the following.
\begin{align*}
	&\mathrm{Tr}\left(\sum_{j=0}^{k+2}\sqrt{-1}\xi\wedge(\sqrt{-1}F_{H})^{j}\wedge \xi^{*}\wedge (\sqrt{-1}F_{H})^{k+2-j}\right)\\
	&-\mathrm{Tr}\left(\sum_{j=0}^{k+1}\sqrt{-1}\xi\wedge(\sqrt{-1}F_{H})^{j}\wedge \xi^{*}\wedge (\sqrt{-1}F_{H})^{k+1-j}\right)\wedge \omega\\
	&=\frac{(k+3)!}{3}(3ps-a^{2})-\mathrm{Tr}\left(\sum_{j=0}^{k+1}\sqrt{-1}\xi\wedge(\sqrt{-1}F_{H})^{j}\wedge \xi^{*}\wedge (\sqrt{-1}F_{H})^{k+1-j}\right)\wedge \omega.
\end{align*}
The first term is calculated in the previous section \ref{section vbma} (we have ignored the volume form part). So we will only calculate the second term. The second term becomes:
\begin{align*}
	&\mathrm{Tr}\left(\sum_{j=0}^{k+1}\sqrt{-1}\xi\wedge(\sqrt{-1}F_{H})^{j}\wedge \xi^{*}\wedge (\sqrt{-1}F_{H})^{k+1-j}\right)\wedge \omega\\
	&=\mathrm{Tr}\left(\sum_{j=0}^{k+1}E_{11}\wedge(\sqrt{-1}F_{H})^{j}\wedge E_{11}\wedge (\sqrt{-1}F_{H})^{k+1-j}\right)\wedge \sqrt{-1}dz^{3}\wedge d\bar{z}^{3} \wedge\omega.
\end{align*}
This can be seen as the coefficient of $m$ in 
\[\mathrm{Tr}\left(E_{11}\big(\sqrt{-1}F_{H}+mE_{11}\sqrt{-1}dz^{3}\wedge d\bar{z}^{3}\big)^{k+2}\right)\wedge\omega.\]
So let us compute the coefficient of $m$. We compute 
\begin{align*}
	&\mathrm{Tr}\left(E_{11}\big(\mathcal{A}+mE_{11}\sqrt{-1}dz^{3}\wedge d\bar{z}^{3}+\sum_{l=1}^{k}\mathrm{Id}\sqrt{-1}dz^{l+3}\wedge d\bar{z}^{l+3}\big)^{k+2}\right)\wedge\omega\\
	&=\binom{k+2}{2}\mathrm{Tr}\left(E_{11}(\mathcal{A}+mE_{11}\sqrt{-1}dz^{3}\wedge d\bar{z}^{3})^{2}\big(\sum_{l=1}^{k}\mathrm{Id}\sqrt{-1}dz^{l+3}\wedge d\bar{z}^{l+3}\big)^{k}\right)\wedge\omega\\
	&+\binom{k+2}{3}\mathrm{Tr}\left(E_{11}(\mathcal{A}+mE_{11}\sqrt{-1}dz^{3}\wedge d\bar{z}^{3})^{3}\big(\sum_{l=1}^{k}\mathrm{Id}\sqrt{-1}dz^{l+3}\wedge d\bar{z}^{l+3}\big)^{k-1}\right)\wedge\omega\\
	&=\binom{k+2}{2}\mathrm{Tr}\left(E_{11}\mathcal{A}^{2}\big(\sum_{l=1}^{k}\mathrm{Id}\sqrt{-1}dz^{l+3}\wedge d\bar{z}^{l+3}\big)^{k}\right)\wedge\omega\\
	&+\binom{k+2}{3}\mathrm{Tr}\left(E_{11}\mathcal{A}^{3}\big(\sum_{l=1}^{k}\mathrm{Id}\sqrt{-1}dz^{l+3}\wedge d\bar{z}^{l+3}\big)^{k-1}\right)\wedge\omega\\
	&+m\times k!\binom{k+2}{2}\mathrm{Tr}\left(2E_{11}(A_{1}+A_{2})\right)\\
	&+m\times k!\binom{k+2}{3}\mathrm{Tr}\left(2E_{11}\big(\{A_{1},A_{2}\}-\{B_{12},B_{12}^{*}\}\big)+E_{11}A_{1}E_{11}A_{2}+E_{11}A_{2}E_{11}A_{1}\right)\\
		&=\binom{k+2}{2}\mathrm{Tr}\left(E_{11}\mathcal{A}^{2}\big(\sum_{l=1}^{k}\mathrm{Id}\sqrt{-1}dz^{l+3}\wedge d\bar{z}^{l+3}\big)^{k}\right)\wedge\omega\\
		&+\binom{k+2}{3}\mathrm{Tr}\left(E_{11}\mathcal{A}^{3}\big(\sum_{l=1}^{k}\mathrm{Id}\sqrt{-1}dz^{l+3}\wedge d\bar{z}^{l+3}\big)^{k-1}\right)\wedge\omega\\
	&+m\left((k+2)!(p+s)+\frac{k\times(k+2)!}{3}(3ps-a^{2})\right).
\end{align*}
We have ignored the volume form part from the coefficient of $m$ in the last two steps. Hence the condition $E$ for the vector bundle $J$-equation is 
\[\frac{(k+3)!}{3}(3ps-a^{2})-(k+2)!(p+s)-\frac{k\times(k+2)!}{3}(3ps-a^{2})=(k+2)!(3ps-a^{2}-p-s).\]
We are done.

\end{proof}		
		\section{Counter Example for the dHYM Equation}
		\label{section dhym}
		The deformed Hermitian-Yang-Mills equation is not yet correctly formulated in mathematics for higher rank vector bundles. However we will work with the equation suggested by  Collins-Yau \cite{collins2018} and call it the deformed Hermitian-Yang-Mills (dHYM) equation. For a metric $H$ on a holomorphic vector bundle over a compact K\"ahler manifold $(X,\omega)$, the dHYM is 
		\[\mathrm{Im}\left(e^{-\sqrt{-1}\hat{\theta}}\big(\omega\otimes\mathrm{Id}_{E}-F_{H}\big)^{n}\right)=0\]
		with 
		\[\int_{X}\mathrm{Tr}\left(\omega\otimes\mathrm{Id}-F_{H}\right)^{n}\in \mathbb{R}_{>0}e^{\sqrt{-1}\hat{\theta}}\]
		and the imaginary part is defined using the metric $H$. \\
		Now we will take  $\hat{\theta}=-\epsilon\theta+\pi\frac{n}{2}$ and $\omega=\frac{\epsilon\theta}{n} \tilde{\omega}$, where $\epsilon,\theta>0$. We put this in the equation 
		\begin{align*}
			&\mathrm{Im}\left(e^{\sqrt{-1}\epsilon\theta-\sqrt{-1}\pi\frac{n}{2}}\big(\frac{\epsilon\theta}{n}\tilde{\omega}\otimes\mathrm{Id}_{E}-F_{H}\big)^{n}\right)=0\\
			& \implies \mathrm{Im}\left(\big(\cos\epsilon\theta+\sqrt{-1}\sin\epsilon\theta\big)e^{-\sqrt{-1}\pi\frac{n}{2}}\big((\sqrt{-1})^{n}(\sqrt{-1}F_{H})^{n}+\epsilon\theta\tilde{\omega}(\sqrt{-1})^{n-1}(\sqrt{-1}F_{H})^{n-1}+O(\epsilon^{2})\big)\right)\\
			&=0.
		\end{align*}
		For small $\epsilon$, we know that $\cos\epsilon\theta=1+O(\epsilon^{2})$ and $\sin\epsilon\theta=\epsilon\theta+O(\epsilon^{3})$. So from above we get
		\begin{align*}
			& \epsilon\theta\left((\sqrt{-1}F_{H})^{n}-\tilde{\omega}(\sqrt{-1}F_{H})^{n-1}\right)+O(\epsilon^{2})=0\\
			&\implies(\sqrt{-1}F_{H})^{n}-\tilde{\omega}(\sqrt{-1}F_{H})^{n-1}+O(\epsilon)=0. 
		\end{align*}
		So, in this setting we can think that the dHYM equation is just perturbation of the vector bundle $J$-equation. 
		Now we will use the counter example for the vector bundle $J$-equation and implicit function theorem to get counter example for the dHYM equation. \\
		We take 
		\[\sqrt{-1}F_{H}=\sum_{i=1}^{n}A_{i}\sqrt{-1}dz^{i}\wedge d\bar{z}^{i}+\sum_{i<j}B_{ij}\sqrt{-1}dz^{i}\wedge d\bar{z}^{j}+\sum_{i<j} B_{ij}^{*}\sqrt{-1}dz^{j}\wedge d\bar{z}^{i}\]
		and 
		\[\tilde{\omega}=\sum_{i=1}^{n}\sqrt{-1}dz^{i}\wedge d\bar{z}^{i}.\]
	So the equation becomes 
		\begin{align*}
			& 1) \alpha:=\binom{n-1}{2}(n-3)! \left(6psx-b^{2}(2s+v-1)-a^{2}(2x+y-1)+2abc-2(ps+sx+px)\right)+D_{11,\epsilon}=0\\
			& 2)\beta:= \binom{n-1}{2}(n-3)! \left(6quy-c^{2}(2q+r-1)-a^{2}(x+2y-1)+2abc-2(qu+uy+qy)\right)+D_{22,\epsilon}=0\\
			& 3)\gamma:= \binom{n-1}{2}(n-3)! \left(6rvz-c^{2}(q+2r-1)-b^{2}(s+2v-1)+2abc-2(rv+vz+rz)\right)+D_{33,\epsilon}=0,
		\end{align*}
		where $D_{11,\epsilon}$ is a function of $p,s,x,v,y,a,b,c,\epsilon$; $\mathcal{D}_{22,\epsilon}$ is a function of $q,u,y,x,r,a,b,c,\epsilon$; $\mathcal{D}_{33,\epsilon}$ is a function of $r,v,z,q,s,a,b,c,\epsilon$ and for $\epsilon=0$, $D_{ii,\epsilon=0}=0$.
		The condition $E$ for the dHYM equation becomes $(n-1)!(3ps-a^{2}-p-s)+a^{2}G(\epsilon)+H(p,s,\epsilon)$, where  $G(\epsilon),H(p,s,\epsilon)$ are functions such that $G(\epsilon=0)=0$ and $H(p,s,\epsilon=0)=0$. Now we will set up for the implicit function theorem. We define 
		\begin{align*}
			&p(t)=2+2t(1-t)\\
			&s(t)=2-t\\
			&a(t)=\sqrt{\frac{3(n-1)!+H(2,1,\epsilon)}{(n-1)!-G(\epsilon)}}t\\
			&y(t)=x(t)=\frac{1}{2} \\ &q(t)=r(t)=2\\
			&b(t)=\sqrt{\frac{t}{5}}\\
			&c(t)=5.5\times \sqrt{\frac{5t}{12}}
		\end{align*}
		Now we have three equations and three unknowns $v,u,z$. \\
		We define 
		\[T:\mathbb{R}\times\mathbb{R}^{3}\rightarrow \mathbb{R}^{3}\]
		by
		\[(\epsilon,v,u,z)=(\alpha,\beta,\gamma).\]
		We know that for $\epsilon=0$ we have a solution. To get solution for small $\epsilon>0$, we will use implicit function theorem. So we compute at $\epsilon=0$,
		\[	\begin{bmatrix}
			\frac{\partial\alpha}{\partial v}& \frac{\partial \alpha}{\partial u} &\frac{\partial \alpha}{\partial z}\\
			\frac{\partial\beta}{\partial v}& \frac{\partial \beta}{\partial u} &\frac{\partial \beta}{\partial z}\\
			\frac{\partial\gamma}{\partial v}& \frac{\partial \gamma}{\partial u} &\frac{\partial \gamma}{\partial z}	\end{bmatrix}=\binom{n-1}{2}(n-3)!\begin{bmatrix}
			-b^{2}  & 0 & 0\\
			0 &6qy-2(q+y)& 0\\
			6rz-2b^{2}-2(r+z)& 0 & 6rv-2(r+v) 
		\end{bmatrix}.\]
		The determinant is 
		\begin{align*}
			&-\left(\binom{n-1}{2}(n-3)!\right)^{3}b^{2}\big(6qy-2(q+y)\big)\big(6rv-2(r+v)\big)\\
			&=-\left(\binom{n-1}{2}(n-3)!\right)^{3}\frac{t(16+20t+100t^{2})}{5},
		\end{align*}
		where we used the fact that for $\epsilon=0$, $v=2+2t+10t^{2}$ is the solution. This shows that we can apply the implicit function theorem for $t\in(0,1]$. So we have solution for all $t\in(0,1]$. To complete the proof we show that at $t=0$  for $\epsilon=0$(that is the vector bundle $J$-equation), the equation is elliptic. Since ellipticity is an open condition, we will be done by taking the path for $t\in [\delta,1]$ for $\delta>0$ sufficietly small. We only need to do the calculation for dimension $=3$ and rank $=3$.  Indeed the condition $E$ for the vector bundle $J$-equation is
		\begin{align*}
			&\mathrm{Tr}\left(\sum_{i=0}^{2}\sqrt{-1}\xi(\sqrt{-1}F_{H})^{j}\wedge\xi^{*}\wedge(\sqrt{-1}F_{H})^{2-j}\right)-\mathrm{Tr}\left(\sum_{i=0}^{1}\sqrt{-1}\xi(\sqrt{-1}F_{H})^{j}\wedge\xi^{*}\wedge(\sqrt{-1}F_{H})^{1-j}\right)\wedge\tilde{\omega},
		\end{align*} 
		where $\xi=a\otimes g$, $0\ne a=ldz^{1}+mdz^{2}+ndz^{3}\in \Sigma^{1,0}(X_{p})$, $0\ne g\in End(E_{p})$. Now we calculate 
		\begin{align*}
			&\mathrm{Tr}\left(\sum_{i=0}^{2}\sqrt{-1}\xi(\sqrt{-1}F_{H})^{j}\wedge\xi^{*}\wedge(\sqrt{-1}F_{H})^{2-j}\right)-\mathrm{Tr}\left(\sum_{i=0}^{1}\sqrt{-1}\xi(\sqrt{-1}F_{H})^{j}\wedge\xi^{*}\wedge(\sqrt{-1}F_{H})^{1-j}\right)\wedge\tilde{\omega}\\
			&=\sqrt{-1}a\wedge\bar{a}\wedge\mathrm{Tr}\left(\sum_{i=0}^{2}g(\sqrt{-1}F_{H})^{j}\wedge g^{*}\wedge(\sqrt{-1}F_{H})^{2-j}\right)\\
			&-\sqrt{-1}a\wedge\bar{a}\wedge\mathrm{Tr}\left(\sum_{i=0}^{1}g(\sqrt{-1}F_{H})^{j}\wedge g^{*}\wedge(\sqrt{-1}F_{H})^{1-j}\right)\wedge\tilde{\omega}.
		\end{align*}
		The above term is positive if and only if the following is a positive definite matrix
		\begin{align*}
			&\lvert l\rvert^{2}\left(\{A_{2},A_{3}\}\otimes\mathrm{Id}+\mathrm{Id}\otimes\{A_{2},A_{3}\}+A_{2}\otimes A_{3}+A_{3}\otimes A_{2}-\mathrm{Id}\otimes A_{2}-A_{2}\otimes\mathrm{Id}-\mathrm{Id}\otimes A_{3}-A_{3}\otimes\mathrm{Id}\right)\\
			&+\lvert m\rvert^{2}\left(\{A_{1},A_{3}\}\otimes\mathrm{Id}+\mathrm{Id}\otimes\{A_{1},A_{3}\}+A_{1}\otimes A_{3}+A_{3}\otimes A_{1}-\mathrm{Id}\otimes A_{1}-A_{1}\otimes\mathrm{Id}-\mathrm{Id}\otimes A_{3}-A_{3}\otimes\mathrm{Id}\right)\\
			&+\lvert n\rvert^{2}\left(\{A_{2},A_{1}\}\otimes\mathrm{Id}+\mathrm{Id}\otimes\{A_{2},A_{1}\}+A_{2}\otimes A_{1}+A_{1}\otimes A_{2}-\mathrm{Id}\otimes A_{2}-A_{2}\otimes\mathrm{Id}-\mathrm{Id}\otimes A_{1}-A_{1}\otimes\mathrm{Id}\right),
		\end{align*}
		where $A_{1}=A_{2}=2\mathrm{Id}$ and $A_{3}=\frac{1}{2}\mathrm{Id}$. The above expression then reduces to
		\begin{align*}
			&\lvert l\rvert^{2}(\mathrm{Id}\otimes\mathrm{Id})+\lvert m\rvert^{2}\left(\mathrm{Id}\otimes\mathrm{Id}\right)+16\lvert n\rvert^{2}\left(\mathrm{Id}\otimes\mathrm{Id}\right)\\
			&=(\lvert l\rvert^{2}+\lvert m\rvert^{2}+16\lvert n\rvert^{2})\mathrm{Id}\otimes\mathrm{Id}.
		\end{align*}
		This is always positive definite. Hence we are done. This proves theorem \ref{thmdhym}.

		\bibliographystyle{plain} 
	\bibliography{paper}
	\textbf{Gao Chen}(chengao1@ustc.edu.cn) and  \textbf{Kartick Ghosh}(kghosh@ustc.edu.cn)\\
		Institute of Geometry and Physics, University of Science and Technology of China\\
		No.$96$, JinZhai Road, Baohe District, Hefei, Anhui, $230000$, P.R.China
	\end{document}